\documentclass[11pt]{amsart}  % Specifies the document style.
\usepackage{amsthm, amsmath, amscd, amssymb, latexsym, stmaryrd, color}%txfonts 
\usepackage[top=3.0cm, bottom=3.0cm, left=3.0cm, right=3.0cm]{geometry}
\theoremstyle{plain}

\newtheorem{theorem}{Theorem}[section]
\newtheorem{lemma}[theorem]{Lemma}
\newtheorem{proposition}[theorem]{Proposition}
\newtheorem{claim}[theorem]{Claim}
\newtheorem{corollary}[theorem]{Corollary}
\newtheorem{conjecture}[theorem]{Conjecture}
\theoremstyle{definition}

\newtheorem{definition}[theorem]{Definition}

\newtheorem{notation}[theorem]{Notation}

\newtheorem{remark}[theorem]{Remark}
\theoremstyle{remark}
\newtheorem*{ack}{Acknowledgement}
\usepackage[mathscr]{eucal}
\usepackage{graphics, graphpap}
\usepackage{array, tabularx, longtable}
\usepackage{color}
\numberwithin{equation}{section}

\def\s{\mathrm{s}}
\def\p{\mathrm{p}}

\def\nt{\mathrm{nt}}

\def\root{\mathrm{root}}

\def\mon{\mathrm{mon}}

\def\top{\mathrm{top}}
\def\GL{\mathrm{GL}}

\title[Topological zeta functions]{Topological zeta functions of complex plane curve singularities} % Declares the document's title.
\author[Q.T. L\^e]{Quy Thuong L\^e}
\address{VNU University of Science, Vietnam National University, Hanoi \newline \indent 
334 Nguyen Trai Street, Thanh Xuan District, Hanoi, Vietnam}
\email{leqthuong@gmail.com}
\thanks{The first author's research is funded by the Vietnam National University, Hanoi (VNU) under project number QG.19.06.}

\author[K.H. Nguyen]{Khanh Hung Nguyen}
\address{Institute of Mathematics, Vietnam Academy of Science and Technology
\newline
\indent 18 Hoang Quoc Viet Road, Cau Giay District, Hanoi, Vietnam}
\email{hunga1hongbang@gmail.com}
\thanks{}


\begin{document}           % End of preamble and beginning of text.
\begin{abstract}
We study topological zeta functions of complex plane curve singularities using toric modifications and further developments. As applications of the research method, we prove that the topological zeta function is a topological invariant for complex plane curve singularities, we give a short and new proof of the monodromy conjecture for plane curves. 
\end{abstract}
\maketitle                 % Produces the title.

\section{Introduction}\label{sec1}
Let $f$ be a non-constant complex function on a smooth complex algebraic variety $X$, and let $X_0$ be its zero locus. In 1992, using an embedded resolution of singularities Denef and Loeser \cite{DL92} introduced the {\it topological zeta function} for $f$. Let $h: Y\to (X,X_0)$ be an embedded resolution of singularities of $X_0$, i.e, a proper morphism $h: Y\to X$ with $Y$ smooth such that the restriction $Y\setminus h^{-1}(X_0)\to X\setminus X_0$ is an isomorphism and $h^{-1}(X_0)$ is a divisor with normal crossings. The exceptional divisors and irreducible components of the strict transform of $h$ are denoted by $E_i$, where $i$ is in a finite set $S$. The multiplicities $N_i$ of $h^*f$ on $E_i$ and the discrepancies $\nu_i-1$ of the Jacobian of $h$ are determined respectively in the formulas $h^{-1}(X_0)=\sum_{i\in S}N_iE_i$ and $K_Y=h^*K_X+\sum_{i\in S}(\nu_i-1)E_i$. For $I\subseteq S$ we write $E_I$ for the intersection $\bigcap_{i\in I}E_i$ and write $E_I^{\circ}$ for the set $E_I\setminus\bigcup_{j\not\in I}E_j$. For a closed point $x$ in $X_0$, we denote $S_x:=\{ i\in S \mid h(E_i)=x\}$. With the function $f$ and the morphism $h$ as above, the associated topological zeta function is defined as follows
$$Z_f^{\top}(s)=\sum_{I\subseteq S}\chi(E_I^{\circ})\prod_{i\in I}\frac{1}{N_is+\nu_i}.$$
It was shown that the function $Z_f^{\top}(s)$ is independent of the choice of $h$ (cf. \cite[Th\'eor\`eme 3.2]{DL92}), and its poles are interesting numerical invariants, which concern the monodromy conjecture. The {\it local} topological zeta function $Z_{f,x}^{\top}(s)$ associated to $(f,x)$ is also defined in the same way where the sum over $I\subseteq S$ is replaced by the sum over $I\subseteq S$ satisfying $I\cap S_x\not=\emptyset$.

\medskip
It is a fact that the monodromy conjecture is one of important problems in singularity theory, algebraic geometry and other branches of mathematics. In Igusa's original version, it is expected to be a bridge that connects geometry and arithmetic of a integer-coefficient polynomial. The topological version was first stated in \cite{DL92} using the topological zeta function.

\begin{conjecture}[Topological monodromy conjecture]\label{topconj}
If $\theta$ is a pole of $Z_f^{\top}(s)$, then $\exp(2\pi i\theta)$ is an eigenvalue of the monodromy of $(f,x)$ for some closed point $x$ in $X_0$.
\end{conjecture}

Up to now, the positiveness of the conjecture has been confirmed only in particular cases, and finding a proof for the general case is still a widely open problem. Any proof for this conjecture can motivate the development of several fields of mathematics. 

\medskip
In this article, we study the local topological zeta function for reduced complex plane curve singularities $(f,O)$ which have no smooth irreducible components, as well as some related problems in a practical method using toric modifications. The first result, Theorem \ref{zeta-general}, describes explicitly $Z_{f,O}^{\top}(s)$ in terms of the simplified extended resolution graph $\mathbf G_{\s}$ of $(f,O)$ defined in \cite{Thuong1}. Namely, 
$$Z_{f,O}^{\top}(s)=\sum_{\mathscr B}\left[\frac{b_1^{\mathscr B}}{(N(P_{\root}^{\mathscr B})s+\nu(P_{\root}^{\mathscr B}))(N(P_1^{\mathscr B})s+\nu(P_1^{\mathscr B}))}+Z_{\mathscr B}(s)\right],$$
with the sum running over non-top bamboos $\mathscr B$ of $\mathbf G_{\s}$. Each vertex of a bamboo $\mathscr B$ is attached with a primitive vector $P_i^{\mathscr B}=(a_i^{\mathscr B}, b_i^{\mathscr B})^t$, and if the vertex $P_i^{\mathscr B}$ of $\mathbf G_{\s}$ is of degree $r_i^{\mathscr B}+1$, we define
$$Z_{\mathscr B}(s)=\sum_{i=1}^{k^{\mathscr B}}\left[\frac{\det(P_i^{\mathscr B},P_{i+1}^{\mathscr B})}{(N(P_i^{\mathscr B})s+\nu(P_i^{\mathscr B}))(N(P_{i+1}^{\mathscr B})s+\nu(P_{i+1}^{\mathscr B}))}-\frac{r_i^{\mathscr B}}{N(P_i^{\mathscr B})s+\nu(P_i^{\mathscr B})}\right].$$
Here, the numbers $N(P_{\root}^{\mathscr B})$, $\nu(P_{\root}^{\mathscr B})$, $N(P_i^{\mathscr B})$ and $\nu(P_i^{\mathscr B})$ concerning the resolution of singularities of $(f,O)$ are also given in Theorem \ref{zeta-general}. Let $\mathscr B_0$ denote the first bamboo of $\mathbf G_{\s}$. The hypothesis on $(f,O)$ mentioned above means that $a_i=a_i^{\mathscr B_0}\geq 2$, $b_i=b_i^{\mathscr B_0}\geq 2$ and $(a_i, b_i)=1$ for all $i$. Remark that if $a_i=1$ or $b_i=1$ for some $i$, $f$ becomes non-convenient via an analytic change of coordinates described in \cite[Lemma 1.3]{LeOk}. In fact, our method also works well in this case, and the restriction of study to the case of reducedness and $a_i\geq 2$, $b_i\geq 2$ and $(a_i, b_i)=1$ for all $i$ is simply to simplify the notation. Indeed, if $a_i=1$ for some $i$, we meet the so-called {\it exceptional integral} vector $(1,b_i)$ which corresponds to the lowest right end edge of the Newton boundary. In this situation, we add an additional weight vector $(1,b_i)^t+(0,1)^t$, which is the new right end vertex. If $a_i=1$ for some $i$, we may face to this situation several times in higher bamboos $\mathscr B$, i.e. $a_j^{\mathscr B}=1$ for some $j$, while if $a_i\geq 2$, $b_i\geq 2$, $(a_i,b_i)=1$ for all $i$, it then follows from \cite{AO} that $a_j^{\mathscr B}\geq 2$, $b_j^{\mathscr B}\geq 2$ and $(a_j^{\mathscr B}, b_j^{\mathscr B})=1$ for all bamboos $\mathscr B$ and all $j$.

\medskip
As an application of Theorem \ref{zeta-general}, we prove that the local topological zeta function is a topological invariant for reduced complex plane curve singularities (Theorem \ref{top-invariant}). This is in fact not a trivial result because one finds in \cite{Aetal} an example of surface singularities with the same topological type but different local topological zeta functions.

\medskip
As another application of Theorem \ref{zeta-general}, we revisit the works by Loeser \cite{Loe88} and Rodrigues \cite{Rod04} on the monodromy conjecture for curves with some new ideas. Namely, with the method computing $Z_{f,O}^{\top}(s)$ we prove Conjecture \ref{topconj} for reduced complex plane curves (Theorem \ref{maintheorem}). This result was already made in \cite{Loe88} and \cite{Rod04}, our contribution is just a new short proof in terms of an explicit performance of the poles of $Z_{f,O}^{\top}(s)$. We follow the track A'Campo and Oka in \cite{AO} and L\^e in \cite{Thuong1, Thuong2} to reach the proof.

%*************************************************

\section{Nondegenerate complex plane curve singularities}\label{sec2}
\subsection{Toric modifications}
Let $N$ be the $2$-latice $\left\{(a,b)^t \mid a,b\in\mathbb{Z}\right\}$, and $N^+$ its positive subgroup $\left\{(a,b)^t\in N \mid a,b\geq 0\right\}$. We consider $N_{\mathbb{R}}=N\otimes\mathbb{R}$ and $N^+_{\mathbb{R}}=N^+\otimes\mathbb{R}$. By definition, a {\it simplicial cone subdivision} $\Sigma^*$ of $N^+_{\mathbb R}$ is a sequence $(T_1,\dots, T_m)$ of primitive weight vectors in $N^+$ such that $\det(T_i,T_{i+1})\geq 1$ for all $0\leq i\leq m$, with $T_0=(1,0)^t$ and $T_{m+1}=(0,1)^t$. A simplicial cone subdivision $\Sigma^*$ is said to be {\it regular} if $\det(T_i,T_{i+1})=1$ for all $0\leq i\leq m$. It is clear that $N^+_{\mathbb{R}}$ is covered by $m+1$ cones $C(T_i,T_{i+1})=\left\{xT_i+yT_{i+1}\mid x,y\geq 0\right\}$ of $\Sigma^*$. These cones are in one-to-one correspondence with the matrices $\sigma_i=(T_i,T_{i+1})$; so we shall identify $C(T_i,T_{i+1})$ with $\sigma_i$ for all $0\leq i\leq m$. 

\medskip
It is a fact that each matrix $\sigma=\left(\begin{matrix}a&b\\ c&d \end{matrix}\right)$ in $\GL(2,\mathbb Z)$ defines a birational map $$\Phi_{\sigma}:\mathbb{C}^2\rightarrow \mathbb{C}^2$$
sending $(x,y)$ to $(x^ay^b,x^cy^d)$. In toric geometry, one uses such birational map to define toric modifications. For a regular simplicial cone subdivision $\Sigma^*$ with vertices $T_1,\dots,T_m$, we consider the cones $\sigma_i=(T_i,T_{i+1})$ and the corresponding toric charts $(\mathbb{C}^2_{\sigma_i};x_i,y_i)$, $0\leq i\leq m$, with $\mathbb{C}^2_{\sigma_i}$ a copy of $\mathbb{C}^2$. On the disjoint union $\bigsqcup_{i=0}^m\left(\mathbb{C}^2_{\sigma_i};x_i,y_i\right)$, as in \cite{Ok2} we consider the equivalence relation given by $(x_i,y_i)\sim (x_j,y_j)$ if and only if $\Phi_{\sigma_j^{-1}\sigma_i}(x_i,y_i)=(x_j,y_j)$. Let $X$ be the quotient of $\bigsqcup_{i=0}^m\left(\mathbb{C}^2_{\sigma_i};x_i,y_i\right)$ by the previous equivalence relation, which is endowed with the quotient topology. Then $X$ is a smooth complex manifold of dimension $2$, with the toric charts $(\mathbb{C}^2_{\sigma_i};x_i,y_i)$ as local coordinates systems. In other words, we can present
$$X=\bigcup_{i=0}^m\left(\mathbb{C}^2_{\sigma_i};x_i,y_i\right),$$ 
where $\mathbb{C}^2_{\sigma_i}$ are viewed as open subsets of $X$, and two charts $(\mathbb{C}^2_{\sigma_i};x_i,y_i)$ and $(\mathbb{C}^2_{\sigma_j};x_j,y_j)$ with nonempty intersection are compatibly glued in such a way that  
\begin{align}\label{luatdan}
(x_i,y_i)\equiv (x_j,y_j)\quad \text{if and only if} \quad (x_i,y_i)\sim (x_j,y_j).
\end{align} 
We now define $\pi: X\to \mathbb C^2$ with $\pi(x_i,y_i)=\Phi_{\sigma_i}(x_i,y_i)$ for $(x_i,y_i)$ in $\mathbb C^2_{\sigma_i}$, $0\leq i\leq m$. This map is compatible with the glueing and it is called the {\it toric modification associated to the regular simplicial cone subdivision $\Sigma^*$}.

\medskip
As explained in \cite{LeOk}, the toric modification $\pi$ can be decomposed as a composition of finitely many quadratic blowups. The divisor $\pi^{-1}(O)$ has simple normal crossings with $m$ irreducible components $E(T_i)$, named as exceptional divisors, for $1\leq i \leq m$. For every $1\leq i\leq m$, the exceptional divisor $E(T_i)$ corresponds uniquely to the vertex $T_i$ of $\Sigma^*$, and it is covered by two charts $\mathbb{C}^2_{\sigma_{i-1}}$ and $\mathbb{C}^2_{\sigma_i}$, with the equations $y_{i-1}=0$ and $x_i=0$ respectively. Therefore, only $E(T_i)$ and $E(T_{i+1})$ intersect for all $1\leq i \leq m-1$, and the the intersections are transversal. The noncompact components $E(T_0)=\{x_0=0\}$ and $E(T_{m+1})=\{y_m=0\}$ are isomorphic to the coordinate axes $x=0$ and $y=0$ respectively.

%***************************

\subsection{A toric resolution for $f(x,y)$}\label{resolution-nondegenerate}
Let $f(x,y)=\sum_{(a,b)\in\mathbb N^2}c_{\alpha\beta}x^{\alpha}y^{\beta}$ be in $\mathbb{C}\{x,y\}$ such that $f(O)=0$. Denote by $\Gamma$ or $\Gamma_f$ the Newton polyhedron of $f(x,y)$. Clearly, the boundary of $\Gamma$ contains finitely many facets each of which is completely defined by a positive primitive weight vector of the form $P=(a,b)^t\in N^+$, where $(a,b)$ is a normal vector of the facet. The singularity $f(x,y)$ at $O$ is said to be {\it nondegenerate with respect to $\Gamma$} if it has the form
\begin{equation}\label{hamf-tongquat}
\begin{aligned}
f(x,y)&=cx^ry^sf_1(x,y)\cdots f_k(x,y),\\ 
f_i(x,y)&=\prod_{\ell=1}^{r_i}(y^{a_i}+\xi_{i\ell}x^{b_i})+\text{(higher terms)},
\end{aligned}
\end{equation}
where $c\not=0$, and for every $1\leq i\leq k$,  
\begin{equation}\label{dieukienxi}
\begin{gathered}
(a_i,b_i)=1,\\
\xi_{i\ell}\not=0,\ \xi_{i\ell}\not=\xi_{i\ell'}\ \text{if}\ \ell\not=\ell'.
\end{gathered}
\end{equation}
For simplicity, we shall assume that $c=1$ and $r=s=0$ in the formula of $f(x,y)$. Then the Newton polyhedron $\Gamma$ has $k$ primitive weight vectors $P_1=(a_1,b_1)^t,\dots, P_k=(a_k,b_k)^t$ as $k$ compact facets. We define an ordering on primitive vectors as follows $P<Q$ if $\det(P,Q)>0$. We order the $P_i$ in such a way that $P_1<\cdots<P_k$. 

\medskip
Let $\Sigma^*$ be a regular simplicial subdivision with vertices $T_j=(c_j,d_j)^t$, $1\leq j\leq m$, augmented by $(c_0,d_0)=(1,0)$, $(c_{m+1},d_{m+1})=(0,1)$, with $\det(T_j,T_{j+1})=1$ for all $0\leq j\leq m$. We say that $\Sigma^*$ is {\it admissible for $f(x,y)$} if $\{P_1,\dots,P_k\}\subseteq \{T_1,\dots, T_m\}$. Let $\pi: X\to \mathbb{C}^{2}$ be the toric modification associated to $\Sigma^*$. Then $\pi$ is said to be {\it admissible for $f(x,y)$} if $\Sigma^*$ is admissible for $f(x,y)$. In this case, $\pi$ is nothing else than a resolution of singularity of $f(x,y)$ at $O$, with simple normal crossing divisors. We respectively denote by $N(T_j)$ and $\nu(T_j)-1$ the multiplicity of $\pi^*f$ and that of $\pi^*(dx\wedge dy)$ on the exceptional divisor $E(T_j)$, for $1\leq j\leq m$. Since the expression of $\pi$ on $\mathbb{C}_{\sigma_j}^2$ is $\pi(x_j,y_j)=(x_j^{c_j}y_j^{c_{j+1}},x_j^{d_j}y_j^{d_{j+1}})$, we have 
$$\pi^*(dx\wedge dy)(x_j,y_j)=x_j^{c_j+d_j-1}y_j^{c_{j+1}+d_{j+1}-1}dx_j\wedge dy_j$$ 
on $\mathbb{C}_{\sigma_j}^2$, thus 
\begin{align}\label{nuj-tongquat}
\nu(T_j)=c_j+d_j,
\end{align}
for all $1\leq j\leq m$. It is clear that if $F$ is an irreducible component of the strict transform of $f(x,y)$, and if $f(x,y)$ is reduced, then $\nu(F)=1$.

\medskip
We are in fact using the ordering defined above by $P<Q$ if $\det(P,Q)>0$. To compute the multiplicity $N(T_j)$ of $\pi^*f$ on $E(T_j)$ we consider the following three cases. The first one is $P_i\leq T_j < P_{i+1}$, for some $1\leq i\leq k-1$. Since $ P_t\leq T_j$ for all $1\leq t\leq i$, it follows from \cite[Section 4.3]{AO} that, on the chart $(\mathbb{C}_{\sigma_j}^2;x_j,y_j)$, and for $1\leq t\leq i$,
\begin{align*}
\pi^*f_t(x_j,y_j)=x_j^{r_tb_tc_j}y_j^{r_tb_tc_{j+1}}\left(\prod_{\ell=1}^{r_t}(x_j^{a_td_j-b_tc_j}y_j^{a_td_{j+1}-b_tc_{j+1}}+\xi_{t\ell})+x_jR_t(x_j,y_j)\right),
\end{align*}
for some $R_t(x_j,y_j)\in \mathbb C\{x_j,y_j\}$. Since $T_j < P_t$ for all $i+1\leq t\leq k$, it follows similarly as previous, for $i+1\leq t\leq k$, that
\begin{align*}
\pi^*f_t(x_j,y_j)=x_j^{r_ta_td_j}y_j^{r_ta_td_{j+1}}\left(\prod_{\ell=1}^{r_t}\left(1+\xi_{t\ell}x_j^{b_tc_j-a_td_j}y_j^{b_tc_{j+1}-a_td_{j+1}}\right)+x_jR_t(x_j,y_j)\right),
\end{align*}
for some $R_t(x_j,y_j)\in \mathbb C\{x_j,y_j\}$. Thus, on the chart $(\mathbb C_{\sigma_j}^2;x_j,y_j)$, 
\begin{align*}
\pi^*f(x_j,y_j)=\prod_{t=1}^i\pi^*f_t(x_j,y_j)\cdot \prod_{t=i+1}^k\pi^*f_t(x_j,y_j)=x_j^{N(T_j)}y_j^{N(T_{j+1})}u(x_j,y_j),
\end{align*}
with $u(x_j,y_j)$ a unit in $\mathbb C\{x_j,y_j\}$, and $N(T_j)=c_j\sum_{t=1}^ir_tb_t+d_j\sum_{t=i+1}^kr_ta_t$. In the same way, for the second case $T_j < P_1$, we get $N(T_j)=d_j\sum_{t=1}^kr_ta_t$, and for the third case $P_k\leq T_j$, we get $N(T_j)=c_j\sum_{t=1}^kr_tb_t$.
Thus, by convention that $P_0:=T_0=(1,0)^t$ and $P_{k+1}:=T_{m+1}=(0,1)^t$, we can summarize the three cases by a common formula as follows 
\begin{align}\label{NTj-nondegenerate}
N(T_j)=c_j\sum_{t=1}^ir_tb_t+d_j\sum_{t=i+1}^kr_ta_t,
\end{align}
where $P_i\leq T_j < P_{i+1}$, for all $1\leq j\leq m$.

\medskip
When $T_j=P_i$ for some $i$, $$\pi^*f_i(x_j,y_j)=x_j^{r_ia_ib_i}y_j^{r_ib_ic_{i+1}}\left(\prod_{\ell=1}^{r_i}(y_j+\xi_{i\ell})+x_jR_i(x_j,y_j)\right),$$ 
with $R_i(x_j,y_j)$ in $\mathbb C\{x_j,y_j\}$. Therefore, there are $r_i$ irreducible components of the strict transform intersecting transversally with $E(P_i)$ at $(0,-\xi_{i\ell})$, $1\leq \ell\leq r_i$, in the chart $(\mathbb C_{\sigma_j}^2;x_j,y_j)$. If $2\leq j\leq m-1$ and $T_j\not=P_i$ for all $1\leq i\leq k$, then $E(T_j)$ intersects with exactly two other exceptional divisors and does not intersect with the strict transform. Also, if $T_1\not=P_1$ (resp. $T_m\not=P_k$), then $E(T_1)$ (resp. $E(T_k)$) intersects with only one divisor. 

The below is the configuration of the toric resolution for the nondegenerate singularity $f(x,y)$ at $O$:

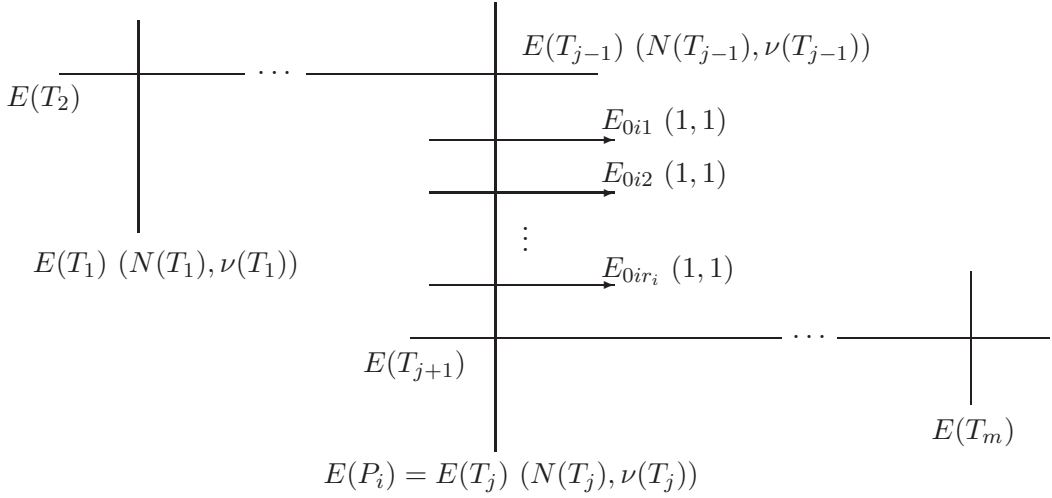
\begin{center}
\begin{picture}(220,185)(-90,-130)
%\put(-120,-12){\line(1,0){70}}
\put(-160,25){\line(1,0){70}}
\put(-85,25){\dots}
\put(-66.5,25){\line(1,0){110}}

\put(-20,0){\vector(1,0){70}}
\put(45,4){$E_{0i1}\ (1,1)$}

\put(-20,-20){\vector(1,0){70}}
\put(45,-16){$E_{0i2}\ (1,1)$}

\put(15,-42){\vdots}

\put(-20,-55){\vector(1,0){70}}
\put(45,-52){$E_{0ir_i}\ (1,1)$}

\put(-180,13){$E(T_2)$}

%E(Ti0+1)
\put(-27,-75){\line(1,0){140}}
\put(117.3,-75){\dots}
\put(-45,-88){$E(T_{j+1})$}

%E(Tm)
\put(134.6,-75){\line(1,0){80}}
\put(185,-100){\line(0,1){50}}

\put(170,-113){$E(T_m)$}

\put(15,32){$E(T_{j-1}) \ (N(T_{j-1}),\nu(T_{j-1}))$}
%-----------
\put(-130,-35){\line(0,1){80}}
\put(-170,-50){$E(T_1) \ (N(T_1),\nu(T_1))$}

%E(P)
\put(5,-118){\line(0,1){170}}
\put(-60,-130){$E(P_i)=E(T_j) \ (N(T_j),\nu(T_j))$}
\end{picture}
\end{center}

%*************************

\subsection{The topological zeta function of a nondegenerate singularity}
Let $f(x,y)$ be a singularity at $O$ nondegenerate with respect to its Newton polyhedron $\Gamma$. Assume that $f(x,y)$ has the form as in (\ref{hamf-tongquat}) and (\ref{dieukienxi}) with $c=1$ and $r=s=0$. Recall that $P_i=(a_i,b_i)^t$ for $0\leq i\leq k+1$, with $(a_0,b_0)=(1,0)$ and $(a_{k+1},b_{k+1})=(0,1)$.

\begin{theorem}\label{zeta-nondegenerate}
With $f(x,y)$ nondegenerate as previous, $Z_{f,O}^{\top}(s)$ equals
$$\sum_{i=0}^k\frac{\det(P_i,P_{i+1})}{(N(P_i)s+\nu(P_i))(N(P_{i+1})s+\nu(P_{i+1}))}-\frac{s}{s+1}\sum_{i=1}^k\frac{r_i}{N(P_i)s+\nu(P_i)},$$
where, for every $0\leq i\leq k+1$, $\nu(P_i)=a_i+b_i$ and $N(P_i)=a_i\sum_{t=1}^ir_tb_t+b_i\sum_{t=i+1}^kr_ta_t$.
\end{theorem}

\begin{proof}
We use the toric resolution described in Section \ref{resolution-nondegenerate} to compute the topological zeta function. Here is the table with the strata $E_I^{\circ}$ of $\pi^{-1}(f^{-1}(O))$ and their Euler characteristic:

\medskip
\begin{center}
\begin{tabular}{|c|c|c|}
\hline 
{Strata} & {Euler char.} & {Conditions} \\ 
\hline
$E(T_1)^{\circ}$, $E(T_m)^{\circ}$ & $1$ & \\
\hline
$E(T_j)^{\circ}$ & $0$ & $1<j< m,\ T_j\not=P_i\ (\forall\ 1\leq i\leq k)$\\
\hline
$E(P_i)^{\circ}$ & $-r_i$ & $1\leq i\leq k$\\
\hline
$E_{0i\ell}^{\circ}$ & $0$ & $1\leq i\leq k$, $1\leq \ell \leq r_i$\\
\hline 
$E_{0i\ell}\cap E_{0i\ell'}=\emptyset$ & $0$ & $1\leq i\leq k$, $\ell \not= \ell'$\\
\hline 
$E(T_j)\cap E(T_{j+1})=1 \text{pt}$ & $1$ & $1\leq j< m$ \\ 
\hline 
$E(T_j)\cap E(T_{j'})=\emptyset$ & $0$ & $|j-j'|\geq 2$ \\  
\hline 
$E(T_j) \cap E_{0i\ell}=\emptyset$ & $0$ & $1\leq i\leq k$, $1\leq \ell\leq r_i$, $T_j\not=P_i\ (\forall\ i)$\\
\hline
$E(P_i)\cap E_{0i\ell}=1 \text{pt}$ & $1$ & $1\leq i\leq k$, $1\leq \ell \leq r_i$\\
\hline
$E(P_i)\cap E_{0i'\ell}=\emptyset$ & $0$ & $1\leq i\not=i'\leq k$\\
\hline
\end{tabular}
\end{center}
By definition, the topological zeta function $Z_{f,O}^{\top}(s)$ is the sum of the following functions 
\begin{gather*}
Z_1=\frac{1}{N(T_1)s+\nu(T_1)}, \ Z_2=\frac{1}{N(T_m)s+\nu(T_m)}, \ Z_3=\sum_{i=1}^k\frac{-r_i}{N(P_i)s+\nu(P_i)}, \\
Z_4=\sum_{j=1}^{m-1}\frac{1}{(N(T_j)s+\nu(T_j))(N(T_{j+1})s+\nu(T_{j+1}))},\ Z_5=\sum_{i=1}^k\frac{r_i}{(s+1)(N(P_i)s+\nu(P_i))}.
\end{gather*}
For all $0\leq i\leq k+1$, let $j_i$ be the index with $0\leq j_i \leq m+1$ and $T_{j_i}=P_i$. Then $Z_4$ equals
$$
\sum_{j=1}^{j_1\!-\!1}\frac{1}{(N(T_j)s\!+\!\nu(T_j))(N(T_{j\!+\!1})s\!+\!\nu(T_{j\!+\!1}))}\!+\!\sum_{j=j_k}^{m\!-\!1}\frac{1}{(N(T_j)s\!+\!\nu(T_j))(N(T_{j\!+\!1})s\!+\!\nu(T_{j\!+\!1}))}
$$
plus
$$\sum_{i=1}^{k-1}\sum_{j=j_i}^{j_{i+1}-1}\frac{1}{(N(T_j)s+\nu(T_j))(N(T_{j+1})s+\nu(T_{j+1}))}.
$$

\begin{claim}\label{bd4.2}
For $0\leq i\leq k$ and $j_i\leq j\leq j_{i+1}-1$,
$$\begin{vmatrix}
N(T_{j+1})&N(T_j)\\ 
\nu(T_{j+1})&\nu(T_j)
\end{vmatrix}=D_i:=\sum_{t=i+1}^kr_ta_t-\sum_{t=1}^ir_tb_t.$$ 
\end{claim}

\medskip
The proof of this claim is trivial, thanks to (\ref{nuj-tongquat}), (\ref{NTj-nondegenerate}). If $D_i\not=0$, then for $j_i\leq j\leq j_{i+1}-1$,
$$\frac{1}{(N(T_j)s+\nu(T_j))(N(T_{j+1})s+\nu(T_{j+1}))}=\frac{N(T_{j+1})/D_i}{N(T_{j+1})s+\nu(T_{j+1})}-\frac{N(T_j)/D_i}{N(T_j)s+\nu(T_j)}.$$
In particular, $D_0$ and $D_k$ are automatically nonzero, since $D_0=N(T_1)$ and $D_k=-N(T_m)$. Moreover, $N(P_1)/D_0=b_1$ and $N(P_k)/D_k=-a_k$, hence we have
\begin{align*}%\label{to:j0-1}
\sum_{j=1}^{j_1-1}\frac{1}{(N(T_j)s+\nu(T_j))(N(T_{j+1})s+\nu(T_{j+1}))}&=\frac{b_1}{N(P_1)s+\nu(P_1)}-Z_1,\\
\sum_{j=j_k}^{m-1}\frac{1}{(N(T_j)s+\nu(T_j))(N(T_{j+1})s+\nu(T_{j+1}))}&=\frac{a_k}{N(P_k)s+\nu(P_k)}-Z_2.
\end{align*}
For $1\leq i\leq k-1$, if $D_i\not=0$, then
\begin{align*}%\label{to:ji-1}
I_i:=\sum_{j=j_i}^{j_{i+1}-1}\frac{1}{(N(T_j)s+\nu(T_j))(N(T_{j+1})s+\nu(T_{j+1}))}&=\frac{N(P_{i+1})/D_i}{N(P_{i+1})s+\nu(P_{i+1})}-\frac{N(P_i)/D_i}{N(P_i)s+\nu(P_i)}\\
&=\frac{\det(P_i,P_{i+1})}{(N(P_i)s+\nu(P_i))(N(P_{i+1})s+\nu(P_{i+1}))}.
\end{align*}
Also, if $D_i=0$, then for $j_i\leq j\leq j_{i+1}-1$ we have
$$\frac{1}{\lambda_j\lambda_{j+1}}=(a_i+b_i)\left(\frac{c_j}{\lambda_j}-\frac{c_{j+1}}{\lambda_{j+1}}\right) \ \text{for} \ \lambda_j:=\frac{N(T_j)}{N(P_i)}=\frac{\nu(T_j)}{\nu(P_i)}=\frac{c_j+d_j}{a_i+b_i};$$
hence
\begin{align*}
I_i=\frac{\det(P_i,P_{i+1})\nu(P_i)/\nu(P_{i+1})}{\left(N(P_i)s+\nu(P_i)\right)^2}=\frac{\det(P_i,P_{i+1})}{(N(P_i)s+\nu(P_i))(N(P_{i+1})s+\nu(P_{i+1}))}.
\end{align*}
In conclusion, by the above computation, $Z_{f,O}^{\top}(s)$ equals
$$\sum_{i=0}^k\frac{\det(P_i,P_{i+1})}{(N(P_i)s+\nu(P_i))(N(P_{i+1})s+\nu(P_{i+1}))}-\sum_{i=1}^k\frac{r_is}{(s+1)(N(P_i)s+\nu(P_i))},
$$
and the theorem is proved.
\end{proof}

We can deduce from the proof of Theorem \ref{zeta-nondegenerate} that $-\frac{\nu(P_i)}{N(P_i)}$ is a pole of order $2$ of the topological zeta function $Z_{f,O}^{top}(s)$ if and only if $D_i=0$. Further, also due to Theorem \ref{zeta-nondegenerate}, we can prove the following proposition. We leave the detailed proof to the reader. 

\begin{proposition}\label{ndtruepoles}
With $f(x,y)$ nondegenerate as previous, for any $1\leq i\leq k$, the rational number $-\frac{\nu(P_i)}{N(P_i)}$ is a pole of $Z_{f,O}^{top}(s)$.
\end{proposition}

\section{General complex plane curve singularities}\label{sect3}
\subsection{Toric resolution tree}\label{resolution-tree}
Let $f$ be a reduced complex plane curve singularity at $O$ which has no smooth irreducible components, and let $C=f^{-1}(0)$. Using toric modifications with centers determined canonically in terms of Tschirnhausen polynomials (see \cite{AO}), Q.T. L\^e \cite{Thuong1} constructs a resolution of singularity of $f$ at $O$ and a resolution graph $\mathbf G_s$ for $(f,O)$. His method allows to arrange the vertices of $\mathbf G_s$ into an ordering so that we can consider $\mathbf G_s$ as a tree. With the help of \cite{Thuong1}, $\mathbf G_s$ is quite simple but still sufficiently strong to describe combinatorially the monodromy zeta function of $(f,O)$. Further, $\mathbf G_s$ is also used in \cite{Thuong2} to formulate a recurrence formula for the motivic Milnor fiber of $(f,O)$. It is shown explicitly in this article that we can also compute the topological zeta function and give a new proof of the monodromy conjecture for plane curves in terms of $\mathbf G_s$. However, to reach to this goal, we have to construct a more complicated graph $\mathbf G$, which is useful for the computation.  

\medskip
Write $f$ as follows  
\begin{equation}\label{f-initial}
f=f_1\cdots f_k, \quad f_i=f_{i1}\cdots f_{ir_i}, \quad f_{i\ell}=f_{i\ell 1}\cdots f_{i\ell r_{i\ell}},
\end{equation}
where for each $(i,\ell,\tau)$, $f_{i\ell\tau}$ is irreducible in $\mathbb{C}\{x\}[y]$ and of the form 
\begin{equation}\label{ff-initial}
f_{i\ell\tau}(x,y)=(y^{a_i}+\xi_{i\ell}x^{b_i})^{A_{i\ell\tau}}+\text{(higher terms)},
\end{equation}
with $\xi_{i\ell}$ being nonzero and distinct. It is clear that $(a_i,b_i)$ is coprime. In this factorization, the (Newton) principal parts of $f_i$ and $f_j$ are weighted homogeneous of different weights for $i\not=j$, the principal parts of $f_{i\ell}$ and $f_{i\ell'}$ are weighted homogeneous of the same weight (this weight corresponds to $(a_i,b_i)$). We assume that 
$$a_i\geq 2, \ b_i\geq 2 \ \text{and}\ (a_i,b_i)=1 \quad \text{for all $1\leq i\leq k$},$$
the assumption guarantees that $f$ has no smooth branches. In fact, if $a_i=1$ or $b_i=1$, one may use an analytic change of coordinates (cf. \cite[Lemma 1.3]{LeOk}) to make $f$ non-convenient, which we do not want to consider. Put 
$$A_i=A_{i1}+\cdots+A_{ir_i},\quad A_{i\ell}=A_{i\ell 1}+\cdots+A_{i\ell r_{i\ell}}.$$ 
Then by \cite[Section 4.3]{AO}, the $A_{i\ell\tau}$-th Tschirnhausen approximate polynomial of $f_{i\ell\tau}(x,y)$ has the form
$$h_{i\ell}(x,y)=y^{a_i}+\xi_{i\ell}x^{b_i}+\text{(higher terms)}.$$
Put $P_i=(a_i,b_i)^t$ for $1\leq i\leq k$. These weight vectors correspond to the compact facets of the Newton polyhedron $\Gamma$ of $f(x,y)$. Suppose that $P_1<\cdots < P_k$. Let $\Sigma^*$ be a regular simplicial cone subdivision with vertices $T_j=(c_j,d_j)^t \in N^+$, for $1\leq j\leq m$, such that $T_1<\cdots<T_m$ and $\{P_1,\dots,P_k\}\subseteq \{T_1,\dots,T_m\}$. We can assume that $T_1\not=P_1$ and $T_m\not=P_k$ (see \cite{Thuong1}). Let $\pi_O$ be the toric modification associated to $\Sigma^*$. Then we construct the first floor of $\mathbf G$ as follows: The vertices correspond to the exceptional divisors $E(T_1),\dots, E(T_m)$ of $\pi_O$, the edges are edges joining $E(T_j)$ with $E(T_{j+1})$, for all $1\leq j\leq m-1$. These vertices and edges form a subgraph $\mathscr B_0$ of $\mathbf G$, which is named as {\it the first bamboo of $\mathbf G$}. By convention, the coordinates $(x,y)$ will be rewritten as $(x_{\mathscr B_0},y_{\mathscr B_0})$.

\medskip
We construct $\mathbf G$ by induction. Assume that $\mathscr B_{\p}$ is a bamboo of $\mathbf G$, which consists of vertices $E(T^{\mathscr B_{\p}}_1),\dots, E(T^{\mathscr B_{\p}}_{m^{\mathscr B_{\p}}})$ with $T^{\mathscr B_{\p}}_1 <\cdots < T^{\mathscr B_{\p}}_{m^{\mathscr B_{\p}}}$. Let $\pi_{\mathscr B_{\p}}: X_{\mathscr B_{\p}}\to \mathbb C^2$ be the toric modification constructing $\mathscr B_{\p}$, and let $f_{\mathscr B_{\p}}(x_{\mathscr B_{\p}},y_{\mathscr B_{\p}})$ be in $\mathbb{C}\{x_{\mathscr B_{\p}},y_{\mathscr B_{\p}}\}$ for which $\pi_{\mathscr B_{\p}}$ is admissible. Note that $X_{\mathscr B_{\p}}$ is covered by the toric charts $(\mathbb C_{\mathscr B_{\p},\sigma_j}^2;x_{\mathscr B_{\p},j},y_{\mathscr B_{\p},j})$, for $1\leq j\leq m^{\mathscr B_{\p}}$, and that, for simplicity, we sometimes write their coordinates by $(x_j,y_j)$ instead of $x_{\mathscr B_{\p},j},y_{\mathscr B_{\p},j}$. Assume that $f_{\mathscr B_{\p}}(x_{\mathscr B_{\p}},y_{\mathscr B_{\p}})$ has the form
\begin{align*}%\label{fB}
f_{\mathscr B_{\p}}(x_{\mathscr B_{\p}},y_{\mathscr B_{\p}})=U_{\mathscr B_{\p}}(x_{\mathscr B_{\p}},y_{\mathscr B_{\p}})x_{\mathscr B_{\p}}^{N^{\mathscr B_{\p}}}\prod_{i=1}^{k^{\mathscr B_{\p}}}\prod_{\ell=1}^{r^{\mathscr B_{\p}}_i}\prod_{\tau=1}^{r^{\mathscr B_{\p}}_{i\ell}}f^{\mathscr B_{\p}}_{i\ell\tau}(x_{\mathscr B_{\p}},y_{\mathscr B_{\p}}),
\end{align*}
where $N^{\mathscr B_{\p}}$ is in $\mathbb N$, $U_{\mathscr B_{\p}}(x_{\mathscr B_{\p}},y_{\mathscr B_{\p}})$ is a unit in the ring $\mathbb{C}\{x_{\mathscr B_{\p}},y_{\mathscr B_{\p}}\}$, and
\begin{align*}%\label{fBir}
f^{\mathscr B_{\p}}_{i\ell\tau}(x_{\mathscr B_{\p}},y_{\mathscr B_{\p}})=(y_{\mathscr B_{\p}}^{a^{\mathscr B_{\p}}_i}+\xi^{\mathscr B_{\p}}_{i\ell}x_{\mathscr B_{\p}}^{b^{\mathscr B_{\p}}_i})^{A^{\mathscr B_{\p}}_{i\ell\tau}}+\text{(higher terms)}
\end{align*}
are irreducible in $\mathbb{C}\{x_{\mathscr B_{\p}},y_{\mathscr B_{\p}}\}$, with $\xi^{\mathscr B_{\p}}_{i\ell}\not=0$ distinct. It follows from \cite[Section 4.3]{AO} that 
$$a_i^{\mathscr B_{\p}}\geq 2, \ b_i^{\mathscr B_{\p}}\geq 2 \ \text{and} \ (a_i^{\mathscr B_{\p}}\geq 2, b_i^{\mathscr B_{\p}}\geq 2)=1 \quad \text{for all}\ 1\leq i\leq k^{\mathscr B_{\p}},$$
because all $a_i$ (for $1\leq i\leq k$) corresponding to $\mathscr B_0$ are greater than or equal to $2$. Notice that when $\mathscr B_{\p}=\mathscr B_0$, we have $U_{\mathscr B_{\p}}(x_{\mathscr B_{\p}},y_{\mathscr B_{\p}})=1$, $N^{\mathscr B_{\p}}=0$, and $f_{\mathscr B_{\p}}$ is nothing but $f$. Put $P^{\mathscr B_{\p}}_i=(a^{\mathscr B_{\p}}_i,b^{\mathscr B_{\p}}_i)^t$ for all $1\leq i\leq k^{\mathscr B_{\p}}$, and assume that $P^{\mathscr B_{\p}}_1<\cdots <P^{\mathscr B_{\p}}_{k^{\mathscr B_{\p}}}$. By the admissibility for $f_{\mathscr B_{\p}}(x_{\mathscr B_{\p}},y_{\mathscr B_{\p}})$ of $\pi_{\mathscr B_{\p}}$, we have $\{P^{\mathscr B_{\p}}_1,\dots, P^{\mathscr B_{\p}}_{k^{\mathscr B_{\p}}}\}\subseteq \{T^{\mathscr B_{\p}}_1,\dots, T^{\mathscr B_{\p}}_{m^{\mathscr B_{\p}}}\}$. The vertices $E(P^{\mathscr B_{\p}}_1), \dots, E(P^{\mathscr B_{\p}}_{k^{\mathscr B_{\p}}})$ are called {\it the principal vertices of $\mathscr B_{\p}$}. By \cite[Section 4.3]{AO}, the $A^{\mathscr B_{\p}}_{i\ell\tau}$-th Tschirnhausen approximate polynomial of $f^{\mathscr B_{\p}}_{i\ell\tau}(x_{\mathscr B_{\p}},y_{\mathscr B_{\p}})$ has the form
$$h^{\mathscr B_{\p}}_{i\ell}(x_{\mathscr B_{\p}},y_{\mathscr B_{\p}})=y_{\mathscr B_{\p}}^{a^{\mathscr B_{\p}}_i}+\xi^{\mathscr B_{\p}}_{i\ell}x_{\mathscr B_{\p}}^{b^{\mathscr B_{\p}}_i}+\text{(higher terms)}.$$ 

\medskip
If $T^{\mathscr B_{\p}}_j=P^{\mathscr B_{\p}}_{i_0}$, the pullbacks $\pi_{\mathscr B_{\p}}^*f_{\mathscr B_{\p}}$ and $\pi_{\mathscr B_{\p}}^*h_{i_0\ell}$ on the chart $(\mathbb C_{\mathscr B_{\p},\sigma_j}^2;x_j,y_j)$ are as follows
$$\pi_{\mathscr B_{\p}}^*f_{\mathscr B_{\p}}(x_j,y_j)=\xi x_j^{N(P^{\mathscr B_{\p}}_{i_0})}y_j^{N(T^{\mathscr B_{\p}}_{j+1})}\left((y_j+\xi^{\mathscr B_{\p}}_{i_0\ell})^{A^{\mathscr B_{\p}}_{i_0\ell}}+x_jR(x_j,y_j)\right)$$
and
$$\pi_{\mathscr B_{\p}}^*h^{\mathscr B_{\p}}_{i_0\ell}(x_j,y_j)=x_j^{a^{\mathscr B_{\p}}_{i_0}b^{\mathscr B_{\p}}_{i_0}}y_j^{c^{\mathscr B_{\p}}_{j+1}b^{\mathscr B_{\p}}_{i_0}}(y_j+\xi^{\mathscr B_{\p}}_{i_0\ell}+x_jR'(x_j,y_j)),$$
for some $\xi$ in $\mathbb C^*$, $R(x_j,y_j)$ and $R'(x_j,y_j)$ in $\mathbb C\{x_j,y_j\}$. Without loss of generality we can (and will) assume that $\xi=1$. By \cite{AO}, in this step, there is a canonical way to change of variables which uses the Tschirnhausen approximate polynomial $h^{\mathscr B_{\p}}_{i_0\ell}$, namely
\begin{equation}\label{T-change}
\begin{cases}
u=x_j\\
v=\pi_{\mathscr B_{\p}}^*h^{\mathscr B_{\p}}_{i_0\ell}/x_j^{a^{\mathscr B_{\p}}_{i_0}b^{\mathscr B_{\p}}_{i_0}}=y_j^{c^{\mathscr B_{\p}}_{j+1}b^{\mathscr B_{\p}}_i}(y_j+\xi_{i_0\ell}+x_jR'(x_j,y_j)).
\end{cases}
\end{equation}
It is easy to obtain the following lemma.

\begin{lemma}\label{T-inverse}
The inverse modification of (\ref{T-change}) is of the form
\begin{equation*}
\begin{cases}
x_j=u\\
y_j=-\xi_{i_0\ell} + (-\xi_{i_0\ell})^{1/c^{\mathscr B_{\p}}_{j+1}b^{\mathscr B_{\p}}_{i_0}}v+R^{''}(u,v),
\end{cases}
\end{equation*}
for some $R^{''}(u,v)$ in $\mathbb C\{u,v\}$.
\end{lemma}

Fix $i_0$ in $\{1,\dots,k^{\mathscr B_{\p}}\}$ and $\ell_0$ in $\{1,\dots,r_i^{\mathscr B_{\p}}\}$. Since $\xi_{i_0\ell_0}\not=0$, it follows from Lemma \ref{T-inverse} that the pullback $\pi_{\mathscr B_{\p}}^*f_{\mathscr B_{\p}}$ is of the following form, in the Tschirnhausen coordinates $(u,v)$,
\begin{align*}
\pi_{\mathscr B_{\p}}^*f_{\mathscr B_{\p}}(u,v)=U'(u,v)u^{N(P^{\mathscr B_{\p}}_{i_0})}\prod_{i=1}^{k'}\prod_{\ell=1}^{r'_i}\prod_{\tau=1}^{r'_{i\ell}}f'_{i\ell\tau}(u,v),
\end{align*}
where $U'(u,v)$ is a unit in $\mathbb C\{u,v\}$, and $$f'_{i\ell\tau}(u,v)=(v^{a'_i}+\xi'_{i\ell}u^{b'_i})^{A'_{i\ell\tau}}+\text{(higher terms)}$$ 
are irreducible in $\mathbb{C}\{u\}[v]$, with $\xi'_{i\ell}\in \mathbb C^*$ distinct. The Newton polyhedron of $\pi_{\mathscr B_{\p}}^*f_{\mathscr B_{\p}}(u,v)$ again gives rise to an admissible toric modification, which constructs a bamboo $\mathscr B$ whose vertices are denoted by $E(T^{\mathscr B}_1),\dots, E(T^{\mathscr B}_{m^{\mathscr B}})$ with $T^{\mathscr B}_1 <\cdots < T^{\mathscr B}_{m^{\mathscr B}}$. In $\mathbf G$, we connect $E(T^{\mathscr B}_1)$ to $E(P^{\mathscr B_{\p}}_{i_0})$ by a single edge, and this edge is taken into account of $\mathscr B$. 

\begin{definition}\label{topbamboo}
The graph $\mathbf G$ is called the {\it toric resolution tree} $\mathbf G$ of $(f,O)$. The bamboo $\mathscr B$ constructed as above is called {\it the successor {\rm (in $\mathbf G$)} of $\mathscr B_{\p}$ at $E(P^{\mathscr B_{\p}}_{i_0})$ associated to $\ell_0$}. The bamboo $\mathscr B_{\p}$ is called {\it the predecessor {\rm (in $\mathbf G$)} of $\mathscr B$}. A bamboo of $\mathbf G$ which has no successor is called a {\it top bamboo} of $\mathbf G$. A bamboo of $\mathbf G$ which is not a top bamboo is called a {\it non-top bamboo} of $\mathbf G$. Let $\mathbf B^{\nt}$ denote the set of all the non-top bamboos of $\mathbf G$.
\end{definition}

\begin{notation}\label{notation}
Since each bamboo $\mathscr B\not= \mathscr B_0$ determines uniquely $P^{\mathscr B_{\p}}_{i_0}$, hence from now on, we denote $P^{\mathscr B}_{\root}:=P^{\mathscr B_{\p}}_{i_0}$. Remark again that $E(P^{\mathscr B}_{\root})$ is not a vertex of $\mathscr B$, it is a vertex of $\mathscr B_{\p}$.
\end{notation}

Remark that every top bamboo has a unique vertex and a unique edge. The number of top bamboos of $\mathbf G$ is nothing else than the number of irreducible components of the singularity $(f,O)$. The below illustrates a toric resolution tree of a plane curve singularity (where the bamboos containing a unique white vertex are top bamboos):

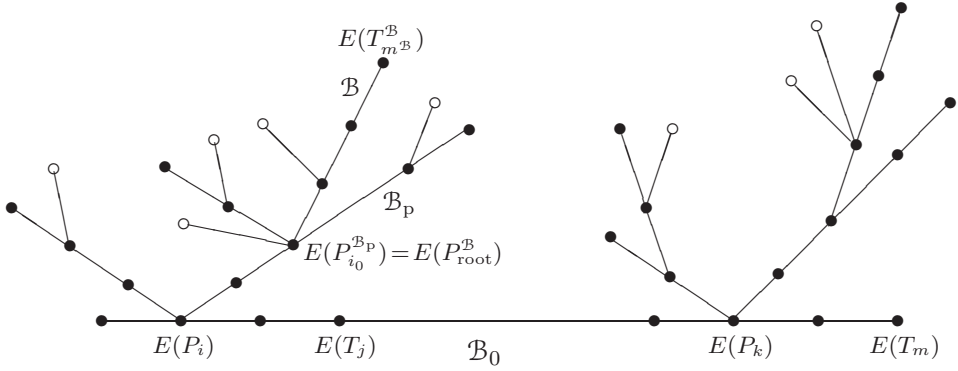
\begin{figure}[ht]
	\centering
	\begin{minipage}{3.1in}
%\begin{center}
\begin{picture}(220,155)(-90,-25)
%\put(-120,-12){\line(1,0){70}}
\put(-130,0){\line(1,0){302}}
\put(10,-15){$\mathscr B_0$}
\put(-109,-12){\footnotesize{$E(P_i)$}}
\put(-48,-12){\footnotesize{$E(T_j)$}}
\put(101,-12){\footnotesize{$E(P_k)$}}
\put(162,-12){\footnotesize{$E(T_m)$}}

\put(-131,-2.5){$\bullet$} 
\put(-101,-2.5){$\bullet$}
\put(-71,-2.5){$\bullet$}
\put(-41,-2.5){$\bullet$}

\put(170,-2.5){$\bullet$}
\put(140,-2.5){$\bullet$}
\put(108,-2.5){$\bullet$}
\put(78,-2.5){$\bullet$}
%--------------------------------
%bamboo B_p
\put(-99,0){\line(3,2){110}}
\put(-22,42){\small{$\mathscr B_{\p}$}}
\put(-52,23){\footnotesize{$E(P_{i_0}^{\mathscr B_{\p}})\!=\!E(P_{\root}^{\mathscr B})$}}

\put(-80,11.5){$\bullet$}
\put(-58.5,26){$\bullet$}
\put(-15,55){$\bullet$}
\put(8,69.5){$\bullet$}

\put(-98,0){\line(-3,2){65}}
\put(-121,11){$\bullet$}
\put(-143,25.5){$\bullet$}
\put(-140.5,28){\line(-1,5){5.7}}
\put(-149.2,55){$\circ$}

\put(-165,40){$\bullet$}
%--------------------------------
\put(-56,28.5){\line(-5,1){39.5}}
\put(-100,34){$\circ$}

\put(-56,29){\line(-5,3){50}}
\put(-83,40.5){$\bullet$}
\put(-80.5,43){\line(-1,5){4.9}}
\put(-88.4,66){$\circ$}

\put(-107,55.5){$\bullet$}

%bamboo B
\put(-56,29){\line(1,2){35}}
\put(-38,85){\small{$\mathscr B$}}
\put(-39,104){\footnotesize{$E(T_{m^{\mathscr B}}^{\mathscr B})$}}

\put(-47.5,49){$\bullet$}
\put(-45,52){\line(-1,1){21.2}}
\put(-70,72){$\circ$}

\put(-36.5,71){$\bullet$}
\put(-24.5,95){$\bullet$}

\put(-13,56){\line(2,5){10}}
\put(-5,80){$\circ$}
%\put(4,103){$\bullet v$}
%------------------------------
\put(111,0){\line(-3,2){47}}
\put(84,14){$\bullet$}
\put(61.5,29){$\bullet$}
\put(86.5,16.5){\line(-1,3){18}}
\put(75,40){$\bullet$}
\put(77,41){\line(1,3){10}}
\put(85,70){$\circ$}

\put(65,70){$\bullet$}

\put(111,1){\line(1,1){80}}
\put(125,15){$\bullet$}
\put(145,35){$\bullet$}
\put(190,80){$\bullet$}
\put(170,60){$\bullet$}

\put(146.8,36){\line(1,3){27}}
\put(154.5,64){$\bullet$}
\put(155,68){\line(-1,1){21}}
\put(156.5,68){\line(-1,3){14}}
\put(130,88){$\circ$}
\put(139.5,109){$\circ$}

\put(163,90){$\bullet$}
\put(171.5,116){$\bullet$}
\end{picture}
%\end{center}
\end{minipage}
\caption{A toric resolution tree of a plane curve singularity}\label{caption1}
\end{figure}

\begin{notation}\label{notation2}
It is convenient to denote
\begin{gather*}
(x_{\mathscr B},y_{\mathscr B}):=(u,v),\ \ f_{\mathscr B}:=\pi_{\mathscr B_{\p}}^*f_{\mathscr B_{\p}}, \ \ U_{\mathscr B}:=U',\\
(a^{\mathscr B}_i,b^{\mathscr B}_i):=(a'_i,b'_i),\ A^{\mathscr B}_{i\ell\tau}:=A'_{i\ell\tau},\ \xi^{\mathscr B}_{i\ell}:=\xi'_{i\ell}, \ k^{\mathscr B}:=k',\ r^{\mathscr B}_i:=r'_i,\ r^{\mathscr B}_{i\ell}:=r'_{i\ell}.
\end{gather*}
Then we rewrite the initial expansion of $f_{\mathscr B}(x_{\mathscr B},y_{\mathscr B})$ as follows
\begin{equation}\label{fB-initial}
\begin{gathered}
f_{\mathscr B}=U_{\mathscr B}x_{\mathscr B}^{N(P^{\mathscr B}_{\root})}\cdot f_1^{\mathscr B}\cdots f^{\mathscr B}_{k^{\mathscr B}}, \quad f_i^{\mathscr B}=f_{i1}^{\mathscr B}\cdots f^{\mathscr B}_{ir_i^{\mathscr B}}, \quad f_{i\ell}^{\mathscr B}=f_{i\ell 1}^{\mathscr B}\cdots f^{\mathscr B}_{i\ell r_{i\ell}^{\mathscr B}},\\
f^{\mathscr B}_{i\ell\tau}(x_{\mathscr B},y_{\mathscr B})=(y_{\mathscr B}^{a^{\mathscr B}_i}+\xi^{\mathscr B}_{i\ell}x_{\mathscr B}^{b^{\mathscr B}_i})^{A^{\mathscr B}_{i\ell\tau}}+\text{(higher terms)}, 
\end{gathered}
\end{equation}
where $a^{\mathscr B}_i\geq 2$, $a^{\mathscr B}_i\geq 2$, $(a^{\mathscr B}_i,b^{\mathscr B}_i)=1$, and $f^{\mathscr B}_{i\ell\tau}(x_{\mathscr B},y_{\mathscr B})$ are irreducible in $\mathbb{C}\{x_{\mathscr B},y_{\mathscr B}\}$, and the complex numbers $\xi^{\mathscr B_{\p}}_{i\ell}$ are nonzero and distinct.
\end{notation}

\begin{notation}
We denote $P^{\mathscr B}_0:=(1,0)^t$, $P^{\mathscr B}_{k^{\mathscr B}+1}:=(0,1)^t$; also, if $\mathscr B=\mathscr B_0$, we write simply $k$ for $k^{\mathscr B_0}$, and $P_i$ for $P_i^{\mathscr B_0}$, for $0\leq i\leq k+1$.
\end{notation}

\begin{remark}\label{remark32}
To a bamboo $\mathscr B$ of $\mathbf G$ we associate a unique bamboo $\mathscr B_{\s}$ whose vertices are the principal vertices of $\mathscr B$ together with $E(T_1^{\mathscr B})$ and $E(T_{m^{\mathscr B}}^{\mathscr B})$. All the edges of $\mathscr B_{\s}$ consist of the one connecting $E(T_1^{\mathscr B})$ with $E(P_1^{\mathscr B})$, the ones connecting $E(P_i^{\mathscr B})$ with $E(P_{i+1}^{\mathscr B})$ for all $1\leq i\leq k^{\mathscr B}-1$, and the one connecting $E(P_{k^{\mathscr B}}^{\mathscr B})$ with $E(T_{m^{\mathscr B}}^{\mathscr B})$. Working with the bamboos $\mathscr B_{\s}$ and using the method in constructing $\mathbf G$ we obtain a tree, which recovers the simplified extended resolution graph $\mathbf G_{\s}$ in \cite{Thuong2}. %We can rename the vertices of $\widetilde{\mathbf G}$ as $E(T_{\l}^{\widetilde{\mathscr B}}), E(P_1^{\widetilde{\mathscr B}}),\dots, E(P_{k^{\widetilde{\mathscr B}}}^{\widetilde{\mathscr B}}), E(T_{\r}^{\widetilde{\mathscr B}})$, since $P_i^{\widetilde{\mathscr B}}=P_i^{\mathscr B}$ for all $\mathscr B$ and $1\leq i\leq k^{\mathscr B}-1$. 
\end{remark}

%--------------------

\subsection{Multiplicities and discrepancies}\label{Mult-Dis}
Let $\mathscr B$ be a bamboo of $\mathbf G$ and $\mathscr B_{\p}$ be the predecessor of $\mathscr B$ in $\mathbf G$. First, using the notation in Section \ref{resolution-tree} (in particular, Notation \ref{notation}) and the same method of computation as in Section \ref{resolution-nondegenerate} we obtain the following lemmas.
\begin{lemma}\label{lem37}
For $\mathscr B=\mathscr B_0$, and $1\leq j\leq m$ with $P_i\leq T_j < P_{i+1}$, we have
\begin{align*}
N(T_j)=c_j\sum_{t=1}^ib_tA_t+d_j\sum_{t=i+1}^ka_tA_t,
\end{align*}
where $A_i=\sum_{\ell=1}^{r_i}\sum_{\tau=1}^{r_{i\ell}}A_{i\ell\tau}$.  
\end{lemma}

\medskip
As above, suppose that $\mathscr B$ has all vertices $E(T^{\mathscr B}_j)$, with $T^{\mathscr B}_j=(c^{\mathscr B}_j,d^{\mathscr B}_j)^t$ and $T^{\mathscr B}_1<\cdots < T^{\mathscr B}_{m^{\mathscr B}}$, and it has $E(P^{\mathscr B}_1),\dots, E(P^{\mathscr B}_{k^{\mathscr B}})$ as the principal vertices.

\begin{lemma}
For $\mathscr B\not=\mathscr B_0$ and $1\leq j\leq m^{\mathscr B}$ with $P^{\mathscr B}_i\leq T^{\mathscr B}_j < P^{\mathscr B}_{i+1}$, we have
\begin{align*}
N(T^{\mathscr B}_j)=c^{\mathscr B}_jN(P^{\mathscr B}_{\root})+c^{\mathscr B}_j\sum_{t=1}^ib^{\mathscr B}_tA^{\mathscr B}_t+d^{\mathscr B}_j\sum_{t=i+1}^{k^{\mathscr B}}a^{\mathscr B}_tA^{\mathscr B}_t,
\end{align*}
where $A^{\mathscr B}_i=\sum_{\ell=1}^{r^{\mathscr B}_i}\sum_{\tau=1}^{r^{\mathscr B}_{i\ell}}A^{\mathscr B}_{i\ell\tau}$.  
\end{lemma}

Consider the Tschirnhausen coordinates $(x_{\mathscr B_{\p}},y_{\mathscr B_{\p}})$, which is used to construct $\mathscr B_{\p}$, and consider the $2$-form $\omega_{\mathscr{B}_{\p}}=dx_{\mathscr{B}_{\p}}\wedge dy_{\mathscr{B}_{\p}}$ on $(\mathbb C^2;x_{\mathscr B_{\p}},y_{\mathscr B_{\p}})$ (note that $(x_{\mathscr B_0},y_{\mathscr B_0})=(x,y)$ and $\omega:=dx\wedge dy$). Let $\pi_{\mathscr B_{\p}}: X_{\mathscr B_{\p}} \to (\mathbb C^2;x_{\mathscr B_{\p}},y_{\mathscr B_{\p}})$ be the toric modification constructing $\mathscr B_{\p}$. Suppose that $j'$ is the index such that $T^{\mathscr B_{\p}}_{j'}=P^{\mathscr B}_{\root}$. Then, in the chart $(\mathbb C^2_{\mathscr B_{\p},j'};x_{\mathscr B_{\p},j'},y_{\mathscr B_{\p},j'})$ of $X_{\mathscr B_{\p}}$, we have 
$$\Phi_{\mathscr B_{\p}}^*\omega=x_{\mathscr B_{\p},j'}^{\nu(P^{\mathscr B}_{\root})-1}y_{\mathscr B_{\p},j'}^{\nu-1}dx_{\mathscr B_{\p},j'} \wedge dy_{\mathscr B_{\p},j'}$$
for some $\nu$ in $\mathbb N^*$, where $\Phi_{\mathscr B_{\p}}$ is the composition of the toric modifications along the series of consecutive bamboos from $\mathscr B_0$ to $\mathscr B_{\p}$ in $\mathbf G$. Via the change of variables in Lemma \ref{T-inverse}, this form $\Phi_{\mathscr B_{\p}}^*\omega$ becomes 
$$\widetilde U(x_{\mathscr B}, y_{\mathscr B})x_{\mathscr B}^{\nu(P^{\mathscr B}_{\root})-1}\omega_{\mathscr{B}},$$
where $\widetilde U(x_{\mathscr B}, y_{\mathscr B})$ is a unit in $\mathbb C\{x_{\mathscr B}, y_{\mathscr B}\}$. Here, due to Notation \ref{notation2}, we replace $(u,v)$ by $(x_{\mathscr B}, y_{\mathscr B})$ when applying Lemma \ref{T-inverse}.

\begin{lemma}\label{discri}
With the previous notation and hypothesis, for $\mathscr{B}=\mathscr{B}_0$ and $1\leq j\leq m$, we have $\nu(T_j)=c_j+d_j$; otherwise, for $1\leq j\leq m^{\mathscr B}$,
$$\nu(T_j^\mathscr{B})=c_j^\mathscr{B}\nu(P^{\mathscr B}_{\root})+d_j^\mathscr{B}.$$ 
\end{lemma}

\begin{proof}
The case $\mathscr B=\mathscr B_0$ is similar as in the nondegenerate case. Now we consider the case $\mathscr B\not=\mathscr B_0$. In the chart $(\mathbb C^2_{\mathscr B,j};x_{\mathscr B,j},y_{\mathscr B,j})$ of $X_{\mathscr B}$, we have
$$\pi^*_{\mathscr{B}}\Big(x_{\mathscr B}^{\nu(P^{\mathscr B}_{\root})-1}\omega_{\mathscr{B}}\Big)=x_{\mathscr B,j}^{c_j^{\mathscr B}\nu(P^{\mathscr B}_{\root})+d_j^{\mathscr B}-1} y_{\mathscr B,j}^{c_{j+1}^{\mathscr B}\nu(P^{\mathscr B}_{\root})+d_{j+1}^{\mathscr B}-1}dx_{\mathscr B,j}\wedge dy_{\mathscr B,j}.
$$
Hence $\nu(T_j^\mathscr{B})=c_j^\mathscr{B}\nu(P^{\mathscr B}_{\root})+d_j^\mathscr{B}$ and the lemma is proved.
\end{proof}

%-------------------
\subsection{The topological zeta function}
Let $f(x,y)$ be a reduced complex plane curve singularity at $O=(0,0)$, in which its initial expansion is given in (\ref{f-initial}) and (\ref{ff-initial}) (with respect to $\mathscr B_0$) and the initial expansion of $f_{\mathscr B}$ in the Tschirnhausen coordinates $(x_{\mathscr B},y_{\mathscr B})$ with respect to $\mathscr B$ is given in (\ref{fB-initial}). The main result can be stated using $\mathbf G_{\s}$ (i.e., only principal vertices of $\mathbf G$) and proved using $\mathbf G$. We use all the notation in Section \ref{Mult-Dis}. Let $\mathbf B$ be the set of the bamboos of $\mathbf G$. Note that we can identify $\mathbf B$ with the set of the bamboos of $\mathbf G_{\s}$.

\begin{theorem}\label{zeta-general}
With the previous notation, put $Z_{\mathscr B}(s)=0$ for $\mathscr B$ being a top bamboo, and
$$Z_{\mathscr B}(s)=\sum_{i=1}^{k^{\mathscr B}}\left[\frac{\det(P_i^{\mathscr B},P_{i+1}^{\mathscr B})}{(N(P_i^{\mathscr B})s+\nu(P_i^{\mathscr B}))(N(P_{i+1}^{\mathscr B})s+\nu(P_{i+1}^{\mathscr B}))}-\frac{r_i^{\mathscr B}}{N(P_i^{\mathscr B})s+\nu(P_i^{\mathscr B})}\right]$$
otherwise, where $\nu(P_i^\mathscr{B})=a_i^\mathscr{B}\nu(P_{\root}^\mathscr{B})+b_i^\mathscr{B}$ and
$$N(P_i^\mathscr{B})=a_i^{\mathscr{B}}N(P_{\root}^\mathscr{B})+a_i^{\mathscr{B}}\sum_{t=1}^i b_t^\mathscr{B} A_t^\mathscr{B}+b_i^\mathscr{B}\sum_{t=i+1}^{k^\mathscr{B}}a^\mathscr{B}_t A^\mathscr{B}_t.$$
Then, the topological zeta function of $(f,O)$ is given by
$$Z_{f,O}^{\top}(s)=\sum_{\mathscr B\in \mathbf B}\left[\frac{b_1^{\mathscr B}}{(N(P_{\root}^{\mathscr B})s+\nu(P_{\root}^{\mathscr B}))(N(P_1^{\mathscr B})s+\nu(P_1^{\mathscr B}))}+Z_{\mathscr B}(s)\right],$$
with $N(P_{\root}^{\mathscr B_0})=0$, $\nu(P_{\root}^{\mathscr B_0})=1$, and $N(P_1^{\mathscr B})=\nu(P_1^{\mathscr B})=b_1^{\mathscr B}=1$ for any top bamboo $\mathscr B$. 
\end{theorem}

\begin{proof}
Let us regard each bamboo $\mathscr B$ of $\mathbf G$ as a subgraph of $\mathbf G$ with the edge connecting $E(T_1^{\mathscr B})$ to $E(P_{\root}^{\mathscr B})$ included. Remark that the vertex $E(P_{\root}^{\mathscr B})$ belongs to the predecessor bamboo $\mathscr B_{\p}$ of $\mathscr B$ in $\mathbf G$, and that each top bamboo consists of a unique vertex and a unique edge.

\medskip
From the definition of $Z_{f,O}^{\top}(s)$, if for each bamboo $\mathscr B$ of $\mathbf G$ which is not a top bamboo, we define $Z'_{\mathscr B}(s)$ as the sum of 
\begin{gather*}
\frac{\delta(\mathscr B)}{N(T_1^{\mathscr B})s+\nu(T_1^{\mathscr B})}, \ \frac{1-\delta(\mathscr B)}{(N(P_{\root}^{\mathscr B})s+\nu(P_{\root}^{\mathscr B}))(N(T_1^{\mathscr B})s+\nu(T_1^{\mathscr B}))}, \ \frac{1}{N(T_{m^{\mathscr B}}^{\mathscr B})s+\nu(T_{m^{\mathscr B}}^{\mathscr B})}, \\ 
\sum_{i=1}^{k^{\mathscr B}}\frac{-r_i^{\mathscr B}}{N(P_i^{\mathscr B})s+\nu(P_i^{\mathscr B})}, \ \text{and} \ Z:=\sum_{j=1}^{m^{\mathscr B}-1}\frac{1}{(N(T_j^{\mathscr B})s+\nu(T_j^{\mathscr B}))(N(T_{j+1}^{\mathscr B})s+\nu(T_{j+1}^{\mathscr B}))},
\end{gather*}
with $\delta(\mathscr B_0)=1$ and $\delta(\mathscr B)=0$ whenever $\mathscr B\not=\mathscr B_0$, and if for each top bamboo $\mathscr B$, we define
$$Z'_{\mathscr B}(s)=\frac{1}{(N(P_{\root}^{\mathscr B})s+\nu(P_{\root}^{\mathscr B}))(s+1)},$$
then $Z_{f,O}^{\top}(s)=\sum_{\mathscr B\in \mathbf B}Z'_{\mathscr B}(s)$. Similarly as in the nondegenarate case (Theorem \ref{zeta-nondegenerate}), we have
$$Z'_{\mathscr B_0}(s)=\sum_{i=0}^k\frac{\det(P_i,P_{i+1})}{(N(P_i)s+\nu(P_i))(N(P_{i+1})s+\nu(P_{i+1}))}-\sum_{i=1}^k\frac{r_i}{N(P_i)s+\nu(P_i)}.$$
Now we consider a bamboo $\mathscr B$ of $\mathbf G$ which is neither the first bamboo $\mathscr B_0$ nor a top bamboo. By the same method of computation as in the proof of Theorem \ref{zeta-nondegenerate} we get
\begin{align*}
Z&=\sum_{i=1}^{k^{\mathscr B}-1}\frac{\det(P_i^{\mathscr B},P_{i+1}^{\mathscr B})}{(N(P_i^{\mathscr B})s+\nu(P_i^{\mathscr B}))(N(P_{i+1}^{\mathscr B})s+\nu(P_{i+1}^{\mathscr B}))}\\
&\quad +\frac{\det(T_1^{\mathscr B},P_1^{\mathscr B})}{(N(T_1^{\mathscr B})s+\nu(T_1^{\mathscr B}))(N(P_1^{\mathscr B})s+\nu(P_1^{\mathscr B}))}+\frac{\det(P_{k^{\mathscr B}}^{\mathscr B},T_{m^{\mathscr B}}^{\mathscr B})}{(N(P_{k^{\mathscr B}}^{\mathscr B})s+\nu(P_{k^{\mathscr B}}^{\mathscr B}))(N(T_{m^{\mathscr B}}^{\mathscr B})s+\nu(T_{m^{\mathscr B}}^{\mathscr B}))}.
\end{align*}
It follows that
\begin{align*}
Z'_{\mathscr B}(s)&=\frac{b_1^{\mathscr B}}{(N(P_{\root}^{\mathscr B})s+\nu(P_{\root}^{\mathscr B}))(N(P_1^{\mathscr B})s+\nu(P_1^{\mathscr B}))}+\frac{a_{k^{\mathscr B}}^{\mathscr B}}{N(P_{k^{\mathscr B}}^{\mathscr B})s+\nu(P_{k^{\mathscr B}}^{\mathscr B})}\\
&\quad +\sum_{i=1}^{k^{\mathscr B}-1}\frac{\det(P_i^{\mathscr B},P_{i+1}^{\mathscr B})}{(N(P_i^{\mathscr B})s+\nu(P_i^{\mathscr B}))(N(P_{i+1}^{\mathscr B})s+\nu(P_{i+1}^{\mathscr B}))}-\sum_{i=1}^{k^{\mathscr B}}\frac{r_i^{\mathscr B}}{N(P_i^{\mathscr B})s+\nu(P_i^{\mathscr B})}.
\end{align*}
Since $a_{k^{\mathscr B}}^{\mathscr B}=\det(P_{k^{\mathscr B}}^{\mathscr B},P_{k^{\mathscr B}+1}^{\mathscr B})$, $N(P_{k^{\mathscr B}+1}^{\mathscr B})=0$, $\nu(P_{k^{\mathscr B}+1}^{\mathscr B})=1$, the theorem is now proved.
\end{proof}

This theorem gives immediately the following corollary.

\begin{corollary}\label{truepoles}
Every pole of $Z^{\top}_{f,O}(s)$ has the form $-\frac{\nu(P_i^\mathscr{B})}{N(P_i^\mathscr{B})}$ for some $\mathscr{B}$ in $\mathbf B$ and some $i$ with $1\leq i\leq k^\mathscr{B}$.
\end{corollary}

In fact, we can go further to state that every number $-\frac{\nu(P_i^\mathscr{B})}{N(P_i^\mathscr{B})}$ is a pole of $Z^{\top}_{f,O}(s)$. However, its proof is rather long while all we need for the proof of the main theorem (Theorem \ref{maintheorem}) is only Corollary \ref{truepoles}. So we skip proving this stronger statement.

\section{Applications of Theorem \ref{zeta-general}}
\subsection{The topological invariance of the zeta function}
%For analytic germs of subsets $(X,x)$ and $(Y,y)$ in $\mathbb C^n$, we say that they are topologically equivalent if there are neighborhoods $U$ of $x$ and $V$ of $y$ in $\mathbb C^n$, and a homeomorphism $\varphi: U\to V$ such that $\varphi(x)=y$ and $\varphi(X\cap U)=Y\cap V$. 
Recall that two analytic function germs $(f,x)$ and $(g,y)$ on $\mathbb C^n$ are topologically equivalent if there are neighborhoods $U$ of $x$ and $V$ of $y$ in $\mathbb C^n$, and a homeomorphism $\varphi: U\to V$ such that $g\circ \varphi=f$. In \cite{Aetal}, Artal Bartolo, Cassou-Nogu\`es, Luengo and Melle Hern\'andez introduce an example which shows that the topological zeta function of a germ of a complex hypersurface singularity is {\it not} a topological invariant of the singularity. However, in this section we shall prove that when $n=2$ the topological zeta function of a complex singularity is exactly a topological invariant.

\begin{theorem}\label{top-invariant}
	For reduced complex plane curve singularities, the local topological zeta function is a topological invariant.
\end{theorem}

\begin{proof}
In the toric resolution tree $\mathbf G$ of the reduced singularity $(f,O)$, consider a sequence of consecutive bamboos from the first one $\mathscr B_0$ to a top one, say $(\mathscr B_0, \mathscr B_1, \dots, \mathscr B_{g+1})$ with $\mathscr B_i$ is the predecessor of $\mathscr B_{i+1}$. Then the sequence of vertices 
$$(P_{\root}^{\mathscr B_1},\dots,P_{\root}^{\mathscr B_{g+1}})$$
corresponds one-to-one to an irreducible component $D$ of $(f,O)$, hence by \cite[Remark 4.5.4]{AO}, to the sequence of Puiseux pairs of the irreducible component of $(f,O)$. Let $D'$ be another irreducible component of $(f,O)$, which corresponds to a sequence of consecutive bamboos $(\mathscr B'_0=\mathscr B_0, \mathscr B'_1, \dots, \mathscr B'_{g'+1})$. Let $\theta$ be the index such that 
$$P_{\root}^{\mathscr B_t}=P_{\root}^{\mathscr B'_t}, \ 0\leq t\leq \theta, \quad \text{and}\ P_{\root}^{\mathscr B_{\theta+1}}\not=P_{\root}^{\mathscr B'_{\theta+1}}.$$ 
Via Notation \ref{notation}, fixing a bamboo $\mathscr B$ of $\mathbf G$ we introduce new notations as follows: If $P_i^{\mathscr B}=(a_i^{\mathscr B},b_i^{\mathscr B})$ is the weight vector in the initial expansion of $\Phi_{\mathscr B}^*D=f_{i\ell\tau}^{\mathscr B}$ (for some $\ell$ and $\tau$), with $\Phi_{\mathscr B}$ defined in the paragraph right before Lemma \ref{discri}, then we put
$$a(P_i^{\mathscr B}):=a_i^{\mathscr B},\quad b(P_i^{\mathscr B}):=b_i^{\mathscr B}, \quad A_D(P_i^{\mathscr B})=A_{i\ell\tau}^{\mathscr B}.$$
By \cite[Lemma 3.4.2]{AO}, the intersection number $I(D,D';O)$ is computed as follows
$$
I(D,D';O)=\sum_{t=0}^{\theta}a(P_{\root}^{\mathscr B_t})b(P_{\root}^{\mathscr B_t})A_D(P_{\root}^{\mathscr B_t})A_{D'}(P_{\root}^{\mathscr B'_t})+I_{\theta+1},$$
where $I_{\theta+1}$ is equal to
$$\min\left\{a(P_{\root}^{\mathscr B_{\theta+1}})b(P_{\root}^{\mathscr B'_{\theta+1}})A_D(P_{\root}^{\mathscr B_{\theta+1}})A_{D'}(P_{\root}^{\mathscr B'_{\theta+1}}),\ a(P_{\root}^{\mathscr B'_{\theta+1}})b(P_{\root}^{\mathscr B'_{\theta+1}})A_D(P_{\root}^{\mathscr B_{\theta+1}})A_{D'}(P_{\root}^{\mathscr B'_{\theta+1}})\right\}
$$
if $\theta<\min\{g,g'\}$, and 
$$I_{\theta+1}=b(P_{\root}^{\mathscr B_{\theta+1}})A_D(P_{\root}^{\mathscr B_{\theta+1}})A_{D'}(P_{\root}^{\mathscr B'_{\theta+1}})$$
if $\theta=g'=\min\{g,g'\}$. This means that the simplified extended resolution graph $\mathbf G_{\s}$ of $(f,O)$ defined in \cite{Thuong2} (see Remark \ref{remark32}) completely determines the Puiseux pairs of all the irreducible components and the intersection numbers of any couple of them. Thus, by Brieskorn \cite{BK}, $\mathbf G_{\s}$ is a topological invariant of the singularity $(f,O)$. 
	
Clearly, the statement in Theorem \ref{zeta-general} can be stated using $\mathbf G_{\s}$ (i.e., using data from the principal vertices of the bamboos $\mathscr B$ of $\mathbf G$). Then the topological zeta function of $(f,O)$ is completely determined by $\mathbf G_{\s}$ of $(f,O)$. Since $\mathbf G_{\s}$ is a topological invariant of $(f,O)$, so is the topological zeta function of $(f,O)$.
\end{proof}

\subsection{A new proof of the monodromy conjecture for complex plane curves}
In 1975, A'Campo introduced in \cite[Theorem 3]{AC1} a celebrated formula computing the monodromy zeta function of an isolated singularity in terms of its embedded resolution. For complex plane curve singularities, a reduced one is always isolated, so we can apply the formula of A'Campo.

\medskip
Let $f(x,y)$ be a complex plane curve singularity at the origin $O$ of $\mathbb C^2$. Its Milnor fiber $F_O$ is the intersection of $f^{-1}(\eta)$ with a small ball around $O$ for $\eta>0$ very small (see Milnor \cite{M}). The complex vector spaces $H^q(F_O,\mathbb C)$ (resp. $H^*(F_O,\mathbb C)$) admit an automorphism $M_O^{(q)}$ (resp. $M_O$) generated by going once around a loop around $O$ with the starting point $\eta$. 

\begin{theorem}\label{maintheorem}
Let $(f,O)$ be a reduced complex plane curve singularity. If $\theta$ is a pole of $Z_{f,O}^{\top}(s)$, then $\exp(2\pi \sqrt{-1}\theta)$ is an eigenvalue of $M_O$.
\end{theorem}  

\begin{proof}
By the Weierstrass preparation theorem, we can assume that $f(x,y)$ is in $\mathbb C\{x\}[y]$. Denote by $n$ the degree of the polynomial $f(x,y)$ in the variable $y$. It is sufficient to consider the poles different from $1$ of $Z_{f,O}^{\top}(s)$. By Corollary \ref{truepoles}, every pole different from $1$ is of the form $-\nu(P_i^{\mathscr B})/N(P_i^{\mathscr B})$ for some $\mathscr B \in \mathbf B^{\nt}$ and some $i$ with $1\leq i\leq k^\mathscr{B}$, where $\mathbf B^{\nt}$ is the set of all the non-top bamboos of $\mathbf G$ (see Definition \ref{topbamboo}, in Figure \ref{caption1} non-top bamboos are bamboos containing black vertices). 

\medskip
The proof is by induction with many steps. The first step is to verify for the case where the number $k=k^{\mathscr B_0}$ of compact facets of $\Gamma_f$ is $\geq 2$. The second one is to do for $k=1$ and the number $r_1=r_1^{\mathscr B_0}$ of successors of $\mathscr B_0$ in $\mathbf G$ is $\geq 2$. Finally, for the case $k=r_1=1$ we prove by induction on $n$.
	
\medskip
Let $\Delta^{(1)}(t)$ be the characteristic polynomial of $M_O^{(1)}$. By Milnor \cite{M}, $\Delta^{(1)}(t)$ is symmetric, hence $\Delta^{(1)}(t)=(1-t)Z_{f,O}^{\mon}(t)$, where $Z_{f,O}^{\mon}(t)$ is the monodromy zeta function of $(f,O)$. We recall the computation of $Z_{f,O}^{\mon}(t)$ in \cite[Theorem 3.5]{Thuong1}, under the light of  \cite[Theorem 3]{AC1}, as follows 
\begin{align}\label{monzeta}
Z_{f,O}^{\mon}(t)=\frac{1}{1-t^{N(T_1)}}\prod_{\mathscr{B}\in\mathbf B^{\nt}}\frac{\prod_{i=1}^{k^{\mathscr{B}}}(1-t^{N(P_i^{\mathscr B })})^{r_i^{\mathscr B}}}{1-t^{N(T_{m^{\mathscr B}}^{\mathscr B})}}.
\end{align}
Notice that $N(T_1)$ and $N(T_{m^{\mathscr B}}^{\mathscr B})$ are independent of $T_1$ and $T_{m^{\mathscr B}}^{\mathscr B}$ for any $\mathscr B$ in $\mathbf B^{\nt}$, because
\begin{align}\label{NPT}
N(P_1)=b_1N(T_1),\quad  N(P_{k^{\mathscr B}}^{\mathscr B})=a_{k^{\mathscr B}}^{\mathscr B}N(T_{m^{\mathscr B}}^{\mathscr B}).
\end{align}
Hence, from (\ref{monzeta}), if $k=k^{\mathscr B_0}\geq 2$, then $Z_{f,O}^{\mon}(t)$ equals
\begin{equation*}
\frac{(1-t^{N(P_1)})^{r_1}}{1-t^{N(T_1)}}\cdot \frac{(1-t^{N(P_k)})^{r_k}}{1-t^{N(T_m)}}\prod_{i=2}^{k-1}(1-t^{N(P_i)})^{r_i}
\end{equation*}
times
\begin{equation*}
\prod_{\mathscr B_0\not=\mathscr{B}\in\mathbf B^{\nt}}\frac{(1-t^{N(P_{k^{\mathscr B}}^{\mathscr B})})^{r_{k^{\mathscr B}}^{\mathscr B}}}{1-t^{N(T_{m^{\mathscr B}}^{\mathscr B})}}\prod_{i=1}^{k^{\mathscr B}-1}(1-t^{N(P_i^{\mathscr B})})^{r_i^{\mathscr B}}.
\end{equation*}
In this formula, observe that the complex numbers $\exp\left(-2\pi\sqrt{-1}\nu(P_i^{\mathscr B})/N(P_i^{\mathscr B})\right)$ are surely eigenvalues of $M_O^{(1)}$ if either $\mathscr B=\mathscr B_0$ and $2\leq i\leq k-1$ or $\mathscr B\not=\mathscr B_0$ and $1\leq i\leq k^{\mathscr B}-1$.
	
\medskip
Also in the case $k\geq 2$, we consider the complex numbers $t_1=\exp\left(-2\pi\sqrt{-1}\nu(P_1)/N(P_1)\right)$ and $t_{k^{\mathscr B}}=\exp\left(-2\pi\sqrt{-1}\nu(P_{k^{\mathscr B}}^{\mathscr B})/N(P_{k^{\mathscr B}}^{\mathscr B})\right)$ for every $\mathscr B$ in $\mathbf B^{\nt}$. By (\ref{NPT}) and the recurrence formula of $\nu(P_i^{\mathscr B})$ in Theorem \ref{zeta-general}, we get
$$t_1^{N(T_1)}=\exp\left(-2\pi\sqrt{-1}(a_1+b_1)/b_1\right)=\exp\left(-2\pi\sqrt{-1}a_1/b_1\right)$$
and
$$t_{k^{\mathscr B}}^{N(T_{m^{\mathscr B}}^{\mathscr B})}=\exp\left(-2\pi\sqrt{-1}(a_{k^{\mathscr B}}^{\mathscr B}\nu(P_{\root}^{\mathscr B})+b_{k^{\mathscr B}}^{\mathscr B})/a_{k^{\mathscr B}}^{\mathscr B}\right)=\exp\left(-2\pi\sqrt{-1}b_{k^{\mathscr B}}^{\mathscr B}/a_{k^{\mathscr B}}^{\mathscr B}\right).$$
Since $a_1, b_1\geq 2$ and $a_{k^{\mathscr B}}^{\mathscr B}, b_{k^{\mathscr B}}^{\mathscr B} \geq 2$ are coprime pairs for every $\mathscr B$ in $\mathbf B^{\nt}$, it implies that $a_1/b_1$ and $b_{k^{\mathscr B}}^{\mathscr B}/a_{k^{\mathscr B}}^{\mathscr B}$ is not in $\mathbb Z$, hence $t_1$ (resp. $t_{k^{\mathscr B}}$) is a zero of 
$$\frac{(1-t^{N(P_1)})^{r_1}}{1-t^{N(T_1)}}\quad (\text{resp.}\ \frac{(1-t^{N(P_{k^{\mathscr B}}^{\mathscr B})})^{r_{k^{\mathscr B}}^{\mathscr B}}}{1-t^{N(T_{m^{\mathscr B}}^{\mathscr B})}}).$$
So $t_1$ and $t_{k^{\mathscr B}}$, for all $\mathscr B$ in $\mathbf B^{\nt}$, are eigenvalues of $M_O^{(1)}$, thus the proof for $k\geq 2$ completes.
	
\medskip
We now consider the case $k=1$, that is, the initial expansion of $f(x,y)$ at $O$ has the form $(y^{a_1}+\xi x^{b_1})^A+(\text{higher terms})$, with $\xi$ in $\mathbb C^*$ and $A$ in $\mathbb N^*$. If $r_1\geq 2$, then by (\ref{NPT}), the same arguments as in the case $k\geq 2$ still holds, and we thus have that
$\exp\left(-2\pi\sqrt{-1}\nu(P_1)/N(P_1)\right)$ is an eigenvalue of $M_O^{(1)}$, where $P_1=(a_1,b_1)^t$. Assume that $r_1=1$. We are going to prove the theorem by induction of the degree $n=a_1A$ of the polynomial $f$ in the variable $y$. Obviously, the theorem holds for $A=1$. Assume that the theorem already holds for every function germ of degree in $y$ less than $n$. Let $\mathscr{B}_1$ be the unique successor of $\mathscr{B}_0$. Since $N(T_1^{\mathscr{B}_1})=A$, the function $Z_{f,O}^{\mon}(t)$ equals
$$
\frac{(1-t^{a_1b_1A})(1-t^A)}{(1-t^{b_1A})(1-t^{a_1A})}\cdot \frac{1}{1-t^{N(T_1^{\mathscr{B}_1})}}\cdot \prod_{\mathscr B_0\not=\mathscr{B}\in\mathbf B^{\nt}}\frac{(1-t^{N(P_{k^{\mathscr B}}^{\mathscr B})})^{r_{k^{\mathscr B}}^{\mathscr B}}}{1-t^{N(T_{m^{\mathscr B}}^{\mathscr B})}}\prod_{i=1}^{k^{\mathscr B}-1}(1-t^{N(P_i^{\mathscr B})})^{r_i^{\mathscr B}}.$$
By (\ref{monzeta}) we get
\begin{align*}
Z_{f,O}^{\mon}(t)=\frac{(1-t^{a_1b_1A})(1-t^A)}{(1-t^{b_1A})(1-t^{a_1A})}Z_{\pi_1^{*}f,O'}^{\mon}(t),
\end{align*}
where $O'$ is the origin of the system of Tschirnhausen coordinates after the toric modification $\pi_1$ admissible for $f$. Clearly, $t_1$ is a root of the polynomial
$$\frac{(1-t^{a_1b_1A})(1-t^A)}{(1-t^{b_1A})(1-t^{a_1A})},$$ 
and the degree of $\pi_1^{*}f$ in $y$ is less than $n$. This completes the proof.
\end{proof}
%************************

\begin{ack}
The first author thanks the Vietnam Institute for Advanced Study in Mathematics (VIASM) for warm hospitality during his visit. 
\end{ack}

%\makeaddress
\end{document}